\documentclass[12pt]{article}
\usepackage[all]{xy}

\usepackage[leqno, centertags]{amsmath}

\usepackage{amsfonts,amsmath}
\usepackage{amssymb,latexsym}
\usepackage[mathscr]{eucal}
\usepackage{graphicx,color}
\usepackage[all]{xy}
\usepackage[dvips]{epsfig}
\usepackage[dvips]{graphicx}
\usepackage{amsthm}
\usepackage{amscd}

\objectmargin+{0.8mm}
\theoremstyle{plain}
\newtheorem{theorem}{Theorem}[section]
\newtheorem{definition}{Definition}
\newtheorem{proposition}{Proposition}

\def\Z{\mathbb Z}

\def\M1{\mc{M}_1}
 
\def\L{\mathcal L}

\def\mc{\mathcal}

\def\xyPro{
\begin{xy}
(0 , -7.5)*{\sigma(e_3)} , (12 , -7.5)*{\sigma(e_4)}, (0 ,6.5)*{\sigma(e_1)}, (12 ,6.5 )*{\sigma(e_2)}, 
{\ar(0 ,-6.5 )*{\hole}; (6 ,-2.5)*{}};
{\ar(12 ,-6.5)*{\hole};(6 ,-2.5)*{}};
{\ar@{={~*=<.5pt>}}(6 , -2.5)*{};(6 ,2.5)*{}};
{\ar(6 , 2.5)*{};(0 , 5.5)*{\hole}};
{\ar(6 , 2.5)*{};(12 , 5.5)*{\hole}}
\end{xy}
}

\def\address#1{\expandafter\def\expandafter\@aabuffer\expandafter
	{\@aabuffer{\affiliationfont{#1}}\relax\par
	\vspace*{13pt}}}
\begin{document}

\title{An elementary construction of Khovanov-Rozansky type link homology}

\author{KENJI ARAGANE}

\date{\empty}

\maketitle

\begin{abstract}
In this article, we give an elementary  construction of $sl(n)$-homological
 invariants of links presented by braid forms. The Euler characteristic of this complex is equal to $sl(n)$ quantum polynomial invariant of link. 

\end{abstract}
\section{Introduction}

We present here a new method of categorifing quantum polynomial invariants.

M. Khovanov and L. Rozansky defined  homological invariants of links whose graded Euler classes are quantum polynomial invariants.
They used the  notion of matrix factorization and resolutions of diagrams.
In this paper, we will define another homology theory which is similar to the Khovanov-Rozansky theory and related 
to the computational method of quantum polynomial invariant.

In Khovanov-Rozansky theory, graded vector spaces are
associated to oriented trivalent graphs.
We will construct a graded module whose generators are   sets of colorings
on graphs.
These colored graphs have representation theoretic meanings.
One can easily confirm that if we  take a field coefficient, our graded vector space is isomorphic to the one Khovanov and Rozansky defined using the matrix factorization.
So, our complex is isomorphic to the Khovanov-Rozansky complex as vector spaces. Similar relations have been already discussed and developed.   
Our theory can be defined for the integer coefficient though Khovanov-Rozansky theory are defined  for fields of characteristic zero. 
We first review quantum polynomial link invariants associated to the vector representation of $U_q (sl(n))$ and its graphical calculus.
Next, we introduce the graded module and construct  complexes. We prove the invariance of its homology under the Reidemeister moves.

Our method may suggest the importance of representation theory of quantum groups
 in categorical link theory 
and the possibility of  categorification of quantum polynomial invariants associated with  other representations of quantum groups.

\section{Review}
\subsection{Quantum invariants}
In this chapter we review the state sum construction of the quantum $sl(n)$ link invariant
associated with the vector representation.

\begin{definition}
Let $L$ be an oriented link in $S^3$ and $D$ be an oriented link diagram of $L$.
A resolution of the diagram $D$ means
a  locally oriented trivalent graph which is
obtained by resolutions
of all crossings of $D$ by either one of the two ways depicted in Fig.~\ref{01-res}. $\Box$
 \begin{figure}[htb]
\begin{center}
\includegraphics[scale=0.8]{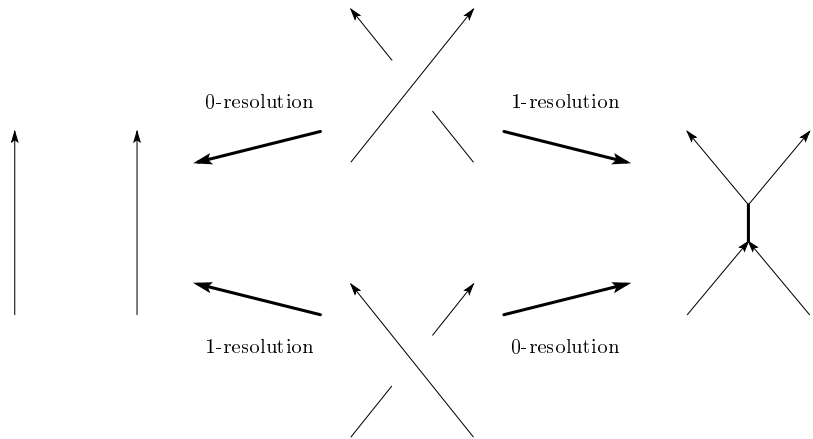}
\end{center}
\caption{0,\ 1-resolutions of crossings}
\label{01-res}
\end{figure}
\end{definition}

We call the wide edge obtained by the 1-resolution of positive crossing and the 0-resolution of negative crossing  the singular edge.
Edges of a resolution of the diagram $D$ except singular  edges  will be  called  the normal edges.
We call  the two edges coming to the singular edge the legs  and  two edges going out the singular edge the heads.


Fix a nonnegative integer $n$ and denote by  
$N$ the set of $n$ elements  $\{ 1,\cdots, n \}$.
For two different elements $a$ and $b$ of $N$ we define a number 
$\pi(a,\ b):=1$ if $a>b$, $\pi(a,\ b):=0$ if $a<b$.
Let $G$ be an oriented, trivalent, planar graph obtained by the resolution of $D$. A state $\sigma$ is an assignment 
of an element  of $N$ to each normal edge $e$.
It should satisfy the  conditions that for each singular edge the set of the elements of $N$ attached to the legs of the singular edge is equal to the set of the elements of $N$ attached to the heads of the singular edge and the elements attached to the heads (or legs) of singular edge  must be different from each other.

\begin{figure}[htb]
\begin{center}

\includegraphics[scale=0.6]{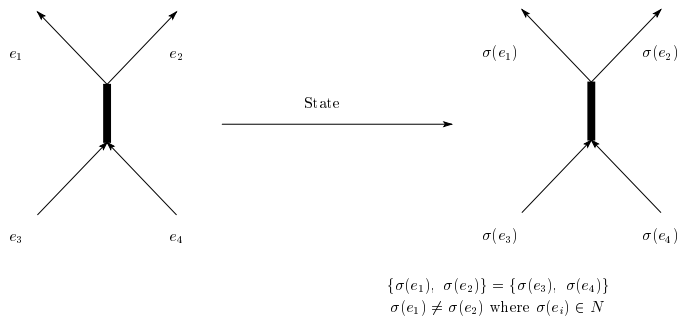}
\end{center}
\caption{Assign an element of $N$ to each normal edge}
\label{StateCondition}
\end{figure} 
Given a state $\sigma$ of $G$, we define the weight  $\mathrm{wt}(v,\ \sigma)$ of a vertex $v$ of a singular edge
to be  
\begin{equation}
\mathrm{wt}( v ,\ \sigma) = q^{1/2-\pi(\sigma(e_1),\ \sigma(e_2))},
\end{equation}
where $q$ is an indeterminate, and  $e_1$ and $e_2$ are the left and right legs (resp. heads) respectively with respect to the orientation of $G$.
\begin{figure}[htb]
\begin{center}
\includegraphics[scale=0.7]{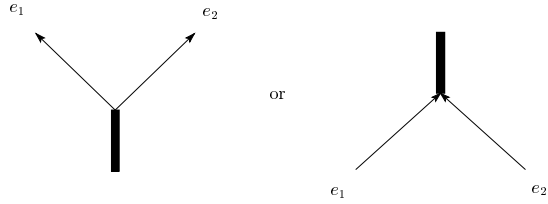}
\caption{Leg and Head of a vertex}
\label{LegHead}
\end{center}
\end{figure}

Let $E$ be a singular edge of $G$ and $v_1,\ v_2$ be  vertices of $E$ as in the figure \ref{ws}.
We define a weight of a singular edge as $wt(E):=wt(v_1)wt(v_2)$.
\begin{figure}[htb]
\begin{center}
\includegraphics[scale=0.7]{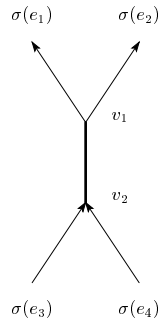}
\caption{Singular edge and vertices}
\label{ws}
\end{center}
\end{figure}
The singular edge can take three values according to the state around the singular edge. 
\begin{gather}
\mathrm{wt} (E)=
\begin{cases}
q & \sigma (e_1)= \sigma (e_3)<\sigma (e_2) = \sigma (e_4), \\
1 & \sigma(e_1) = \sigma (e_4),\  \sigma (e_2) = \sigma (e_3), \\
q^{-1} & \sigma (e_1)= \sigma (e_3)>\sigma (e_2) = \sigma (e_4).
\end{cases}
\end{gather}

If we delete every singular edge of $G$ and identify the heads and the legs of the singular edges which were connected by  singular edges( in other words, collapse the singular edges to  points), we obtain an union of oriented closed curves each of which is equipped with the element of $N$.
\begin{figure}[htb]
\begin{center}
\includegraphics{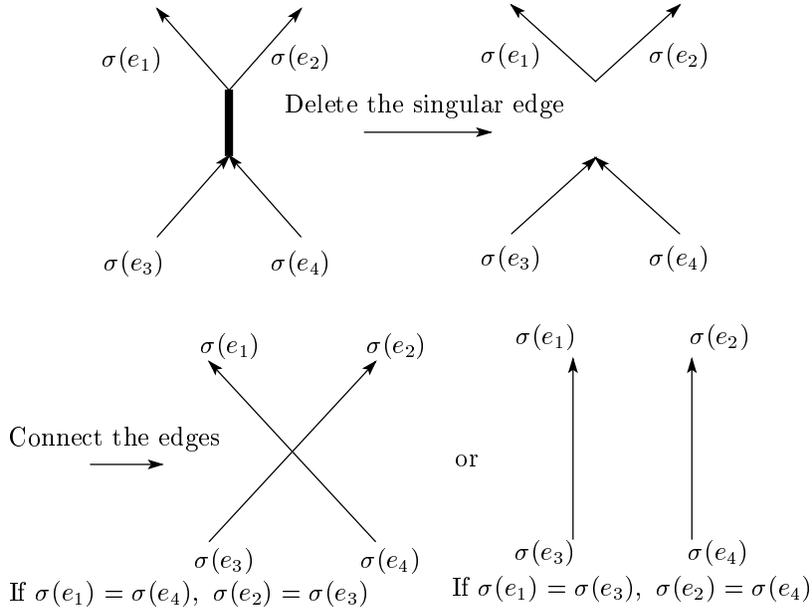}
\end{center}
\caption{Delete and connect the edge}
\label{DeleteAndConnect}
\end{figure}
Then we define the rotation number of the state to be 
$$
\mathrm{rot}(\sigma)= 
\sum_c{(2\sigma(C)-n-1)\mathrm{rot}(C)}
$$
where 
the sum is over all  closed curves  $C$ equipped with $\sigma(C)\in N$ and $\mathrm{rot}(C)$ 
is the rotation number of $C$. (It is $1$ if $C$ is counterclockwise and $-1$ otherwise.)
Now we define a Laurent polynomial${\langle G \rangle}_n$ as  follows.v
$$
{\langle G \rangle}_n=\sum_{states\ \sigma}\{\prod_{vertices\ v}
\mathrm{wt}(v,\ \sigma)q^{\mathrm{rot}(\sigma)}\}
$$

Denote by $Cr$ the set of crossings of $D$. 
Given a crossing of a diagram, we can resolve it in two possible ways.
A $resolution$ of $D$ is a resolution of each crossing of $D$. Thus, $D$ admits $2^{\# \{Cr\}}$ resolutions. There is a one-to-one correspondence between resolutions of $D$ and subsets of the set $Cr$ of crossings. Namely, to $cr \subset Cr$ we  associate a resolution of $D$, denoted $D(cr)$, by taking a 1-resolution of each crossing that belongs to $cr\subset Cr$ and a 0-resolution if the crossing does not lie in $cr$. 
Let $N_0 ^+,\ N_1^+$ be the number of 0-resolutions, 1-resolutions  of positive crossings respectively, and $N_0 ^-,\ N_1 ^-$ be the number of 0-resolutions, 1-resolutions  of negative crossings  respectively.
We denote the writhe number of $D$ as $wr(D)$.
Following these notations, we can state the theorem \cite{5}. These definitions are just based on the representation theory of quantum groups.
\begin{theorem}[Murakami, Ohtsuki, Yamada]
Let $L$ be an oriented link and $D$ be its diagram.
Associate a polynomial to $D$ as follows
$$
{\langle D \rangle}_n := q^{n(-wr(D)) }\sum_{cr\subset Cr} (-1)^{\#\{cr\}} { q^{N_0 ^+ -N_1^-}\langle D(cr) \rangle}_n,
$$

then ${\langle D \rangle}_n$ is invariant under the Reidemeister moves $I,\ II,\ and\ III$.
\end{theorem}
\section{Homology}
\subsection{Cube structure}
We review the cube structure according to \cite{1}.
Let $A$ be a finite set. Denote by $r(A)$ the set of all pairs $(B ,\
 a)$ where $B$ is a subset of $A$ and $a$ an element of $A$ that does not
 belong to $B$. To simplify notation we often denote \\
(a) a one element set $\{a\}$ by $a$,\\
(b) a finite set  $\{a,\ b \cdots c\}$ by $ab\cdots c$,\\
(c) the disjoint union $A\sqcup B$ of two sets $A$, $B$ by $AB$; for example, we denote by $Aa$ the disjoint union of a set $A$ and a one
 element set $\{a\}$; similarly, $Aab$ means $A\sqcup \{a\} \sqcup \{b\}$, and so
 on.\\
\\
\begin{definition}Let $A$ be a finite set and $C$ an additive category. A
 commutative $A$-cube $V$ over $C$ is a collection of objects $V(X)\in \
 Ob(C)$ for each subset $X$of $A$ and morphisms for each $(X,\ a)\in
 r(A)$,\[
\xymatrix{
{\xi}_{a}^{V}(X): V(X) \ar[r] & V(Xa)\\ 
}\]
such that for each triple $(X,\ a,\ b)$ ,where $X$ is a subset of
 $A$ and $a,\ b,\ a\neq b$ are two elements of $A$ that do not lie in
 $X$, there is an equality of morphisms ${\xi}_{b}^{V}(Xa){\xi}_{a}^{V}(X)={\xi}_{a}^{V}(Xb){\xi}_{b}^{V}(X)$
\\
We say a commutative $A$- cube is an $A$-cube. Maps ${\xi}_{a}^{V}$ are
 called structure maps of $V$.
\end{definition}

\begin{definition}Let $A$ be a finite set and $C$ an additive category. A
 skew-commutative $A$-cube $V$ over $C$ is a collection of objects
 $V(X)\in Ob(C)$ for each subset $X$ of $A$, and morphisms
 \[
\xymatrix{
{\xi}_{a}^{V}(X): V(X) \ar[r] & V(Xa)\\ 
}\]
such that for each triple $(X,\ a,\ b)$, where $X$ is a subset of
 $A$ and $a,\ b,\ a\neq b$ are two elements of $A$ that do not lie in
 $X$, there is an equality of morphisms
${\xi}_{b}^{V}(Xa){\xi}_{a}^{V}(X)+{\xi}_{a}^{V}(Xb){\xi}_
{b}^{V}(X)=0$.
\end{definition}

Let $R$ be an arbitrary  commutative ring and we will  consider only cube structure over an additive category $R-mod$ unless stated otherwise.
Given $A$-cubes or skew $A$-cubes $V$ and $W$ over $R$-mod, their tensor
product is defined to be an $A$-cube(if $V$ and $W$ are both cubes or
skew cubes)or a skew cube(if one of $V$, $W$ is a cube and the other is a
skew cube), denoted $V\otimes W$, given by \\
\ $(V\otimes W)(X)=V(X)\otimes W(X)$\\
\ ${\xi}_{a}^{V\otimes W}(X)={\xi}_{a}^{V}(X) \otimes {\xi}_{a}^{W}(X)$\\
where tensor products are taken over $R$.

For a finite set $\L$, denote by $o(\L)$ the set of complete orderings 
  or elements of $\L.$ For $x,\ y\in o(\L)$ let $p(x,\ y)$ be the parity 
 function. $p(x,\ y)=0$ if $y$ can be obtained by  $x$ via  an even number 
  of transpositions of two neighboring elements in the ordering, 
 otherwise, $p(x,y)=1$. 
 To a finite set $\L$, associate a $R$-module $E(\L)$ defined 
  as the quotient of the  $R$-module, freely generated by elements 
  $x$ for all $x\in o(\L)$,  by relations $x = (-1)^{p(x,\ y)}y $ for 
  all pairs $x,y\in o(\L)$.  The module $E(\L)$ is a free  $R$-module 
  of rank $1.$ 
  For $a\not\in \L$ there is a canonical isomorphism of graded $R$-modules 
  $E(\L) \to E(\L   a)$ induced by the map  $o(L)\to o(L  a)$
  that takes $x\in o(L)$ to $xa\in o(L  a).$ Moreover, for 
  $a,\ b,\  a\not= b,$  the diagram below anticommutes 
 
\[
\xymatrix{
E(\L) \ar[r] \ar[d] &E(\L  a) \ar[d]\\
E(\L  b)  \ar[r] & E(\L  a  b)
}\]
  Denote by $E_{I}$ the skew $I$-cube with $E_{I}(\L)= E(\L)$ for 
  $\L \subset \ I$ and the structure map $E_{I}(\L)\to E_{I}(\L   a)$ 
  being canonical isomorphism $E(\L) \to E(\L   a).$ 

  We will use $E_{I}$ to pass 
  from $I$-cubes over $R$-mod to skew $I$-cubes 
  over $R$-mod by tensoring an $I$-cube with $E_{I}.$
 
 Let $V$ be a skew $I$-cube over an abelian category $C.$ 
To $V$ we associate a complex $C(V)=(C^i(V),d^i),\  i\in Z$ 
of objects of $C$ by 
\begin{equation} 
C^i(V)=\oplus_{\L\subset I, |\L|=i} V(\L)
\end{equation} 
The differential $d^i:C^i(V)\to C^{i+1}(V)$ is given on an element 
$x\in V(\L), |\L|=i$ by 
\begin{equation} 
d^i (x)=\sum_{a\in I\setminus \L} \xi^V_a(\L)x.
\end{equation} 
In practice we shall habitually drop $E(\L)$ from the notations of skew cube and cube complex which are derived  from commutative cube because they are only a choice of sign.

\subsection{Colored graph spaces  and constructions of structure morphisms}
We  give a definition of our chain complex of link diagram using
cube structure. 
From now on, we shall always assume that $L$ means a link and  $D$  mean a  link diagram of the closure of a clockwise  oriented  braid representing $L$.
By restricting our attention to braid closure diagram, we will prove the invariance under the Markov-moves instead of the Reidemeister-moves.

\begin{definition}
Let $G$ be an oriented trivalent graph
obtained by a resolution of a link diagram.
We define a graded 
$\Z$-module $C(G)$ as follows. As a generator of $C(G)$ we take  all the states of  graph $G$ and call them as colored graphs.
Colors of edges are represented by  integers from $1$ to $n$.
Fix a colored graph $\sigma\in C(G)$, we will  define  a grade of a state around a singular edge as follows.
Let $E$ be a singular edge of $G$ as in figure~\ref{ws}. 
\begin{gather}
\mathrm{deg} (\ E \ )=
\begin{cases}
1 & \sigma (e_1)= \sigma (e_3)<\sigma (e_2) = \sigma (e_4), \\
0 & \sigma(e_1) = \sigma (e_4),\  \sigma (e_2) = \sigma (e_3), \\
-1 & \sigma (e_1)= \sigma (e_3)>\sigma (e_2) = \sigma (e_4).
\end{cases}
\end{gather}
Similarly,  we  define a grade of parallel edges obtained by 0-resolution of positive crossing as $1$ and also define the degree of parallel  edges obtained by  1-resolution of negative crossing as $-1$.

We define a grading of a colored graph  by 
$$
\mathrm{gr} (\sigma) := {N_0 ^+ - N_1^- -wr(D)\cdot n } +\sum_{singular\ edges\ E} {\mathrm{deg} (E) } + \sum_{curves \ C} {(2\sigma(C)-n-1)} rot(C)
$$
where the sums are  over all singular edges in $G$ and curves as in the definition of the rotation number of the state, $N_0 ^+,\ N_1^-$ are the number of 0-resolutions of positive crossings, 1-resolutions of negative crossings respectively. 
We call this graded module $C(G)$
as colored graph space associated to $G$.
\end{definition}

We will depict the colored graph around the singular edge whose grading is $0$ as in the left hand side of the figure~\ref{degpn} and the colored graph whose grading is not $0$ as in the right hand side of figure~\ref{degpn} . In the figure~\ref{degpn}, we depict the singular edge whose grade is equal to $0$ as a simple crossing but  we will treat as an singular edge.  
\begin{figure}[htb]
\begin{center}
\includegraphics[scale=0.8]{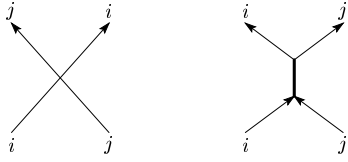}
\caption{States around the singular edge whose grade=0 or $\ne 0$}
\label{degpn}
\end{center}
\end{figure} 
{

\begin{figure}[htb]
\begin{center}
\includegraphics[scale=0.6]{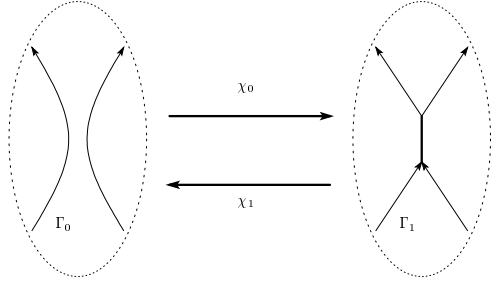}
\caption{Morphisms between
$\Gamma^0$ and $\Gamma^1$
}
\label{fig}
\end{center}

\end{figure}
We will associate a $Cr$-cube to $D$ as follows. 
Given a crossing of a diagram $D$, we can resolve it in two possible ways as in Section2.
A $resolution$ of $D$ is a resolution of each crossing of $D$. Thus, $D$ admits $2^{\# \{Cr\}}-$resolutions. There is a one-to-one correspondence between resolutions of $D$ and subsets $cr$ of the set of crossings  $Cr$ . Namely, to $cr \subset Cr$ we  associate a resolution of $D$, denoted $D(cr)$, by taking a 1-resolution of each crossing that belongs to $cr$ and a 0-resolution if the crossing does not lie in $cr$ and assign the colored graph space $C(D(cr))$ to  $D(cr)$. 

Take a crossing $a$ of $D$ and a subset $A$ of $Cr$ which does not contain $a$.  Resolve crossings of $D$  by the  $1$-resolutions of the  crossings $A$ and the $0-$resolutions of the complement of $A$ except $a$. If we take a resolution of $a$ which replace the crossing $a$ to parallel normal edges, we will denote this resolution of $D$ as $\Gamma^0$. If we take a resolution of $a$ which replace the crossing $a$ to  a singular edge, we will denote this resolution of $D$ as $\Gamma^1$. The difference between $\Gamma^0$ and $\Gamma^1$ is depicted in the figure~\ref{fig}.
To define the cube structure,  we must define morphisms between them.
We define degree $0$ morphisms $\chi_0$ , $\chi_1$ 
between colored graph spaces as follows. Let $E$ be a singular edge of $\Gamma^1$ depicted  in the figure\ref{fig}. 
We will call  colorings of singular edges as positive type (resp. negative)  if  the degree of the singular edge is equal to $1$  (resp. $-1$).
We define the morphisms $\chi_0$, $\chi_1$ as follows.

First we will define  $\chi_1$ morphism. $\chi_1$ morphisms are morphisms from $0-$resolution  to $1-$resolution of a negative crossing.  Take a colored graph $\sigma \in C(\Gamma^1)$ and depict $\sigma$ as a colored oriented trivalent planar graph.
If we delete  all the  singular edges of $\sigma$ and identify the vertices of heads and legs of same singular edge  as in the Fig.~\ref{DelGlue} , 
\begin{figure}[htb]
\begin{center}
\includegraphics[scale=0.8]{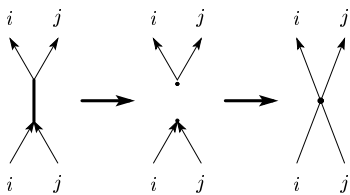}
\end{center}
\caption{Delete the singular edge and identify the vertices }
\label{DelGlue}
\end{figure}
   we will obtain a new colored oriented  graph. We will  call this new graph  as colored circles of $\sigma$ and consider them as  simple closed oriented curves with colorings.  We can also regard the colored circles  as the union of oriented colored normal edges of $\sigma$ because we  deleted only the singular edges of $\sigma$.  We will say that a colored circle $C$ goes through a singular edge $S$ ( or a singular edge $S$ is on a  colored circle $C$),  if a normal edge of the colored circle is the head or the leg of the singular edge $S$.

Take a colored circle $C$ and  singular edges $S$, $S'$ on $C$. We can take an oriented path  on $C$ which is outgoing from $S$ and incoming to $S'$. We will call this subgraph of $C$ as colored path from $S$ to $S'$ and call  these singular edges  a source singular edge $S$ and  a target singular edge $S'$.
We can regard a colored path on $C$ as a subset  of  normal edges of $C$. (We do not assume  $S\ne S'$. If $S=S'$, the colored path from $S$ to $S'$ is equal to  the   colored circle itself. )

If there exist two distinct colored paths which start from the same singular edge and end at the same singular edge, they consist a circle and can be seen as an oriented  colored planar subgraph. If this subgraph does not have self-intersections and does not have  a path which crosses  this subgraph  whose intersection points with this subgraph are positive singular edges or negative  singular edges, we call them as distinguished circle.  
By definition, it is trivial that the normal edges of distinguished circle have just  $2$ distinct colorings. Exchanging the colorings of colored paths of distinguished circle will introduce a new distinguished circle and  a new  colored oriented graph. We will call this coloring exchanging procedure as exchanging   
 colorings  of distinguished circle (or short, color exchange ). 

Suppose the singular edge $E$ of $\Gamma^1$ in the Fig.~\ref{fig} has degree$=1$ and there exist   distinguished circles which contain the leg or the head of $E$ and its source and target singular edges are not positive types. 
If   color exchanging of the distinguished circle  shifts the degree of the source and the target vertices $+1$ and  does not change  the  degrees of singular edges on the distinguished circle   except $3$ singular edges  "the source singular edge"  "the target singular edge" and $E$, we will exchange the colorings of colored paths of  the distinguished circle then replace the singular edge $E$ to the parallel normal edges as in the Fig.~\ref{replace}. We can   consider it as an element of $C(\Gamma^0)$.

\begin{figure}[htb]
\begin{center}
\includegraphics[scale=0.6]{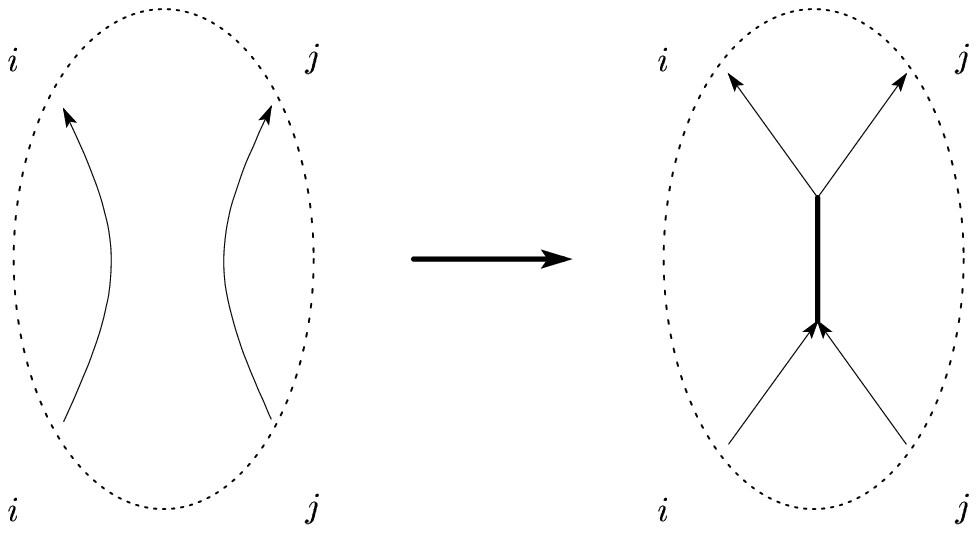}
\end{center}
\caption{Replace}
\label{Rep}
\end{figure}

We  will modify and add other term as the special case. We must add and consider exceptional terms to $\chi_1$  which is related to the Reidemeister1-move. 
If the normal edge which is the head or leg of $E$ does not intersect with 
 other normal edges (i.e., the head and the leg of singular edge $E$ consist a simple loop),  delete the singular edge and  replace to parallel normal  edges  and change the coloring of simple closed curve as to preserve the grade as in the Fig.~\ref{exceptional}.

\begin{figure}[htb]
\begin{center}
\includegraphics[scale=0.6]{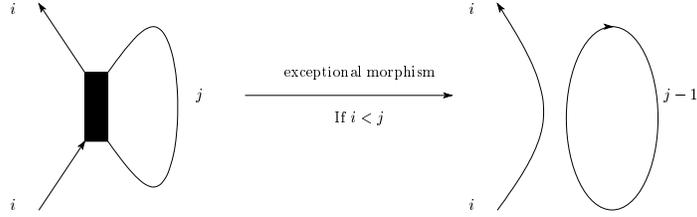}
\end{center}
\caption{Exceptional term}
\label{exceptional}
\end{figure}

If the subgraph does not permit these procedure, we send them to $0$. If there are some possibilities of taking distinguished circles and exchanging the colorings, sum up all of them. If there are no possibilities of taking distinguished circles and exchanging the colorings which satisfies the conditions, we will send  to $0$.

In the case when the singular edge $E$ has  degree$=-1$,  replace the singular edge $E$ to  parallel normal edges  and we can naturally consider colorings of them as in the Fig.~\ref{replace}. We will  consider them as an element of $C(\Gamma^0)$.  We will denote it as $\tilde {\sigma}$ and consider as an element of  $C(\Gamma^0)$. The relation between $\sigma$ and $ \tilde {\sigma }$ is   $\mathrm{gr} (\sigma)= \mathrm{gr} (\tilde {\sigma } ) $.

\begin{figure} [htb]
\begin{center}
\includegraphics[scale=0.8]{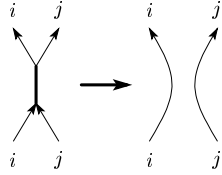}
\end{center}
\caption{Replace the singular edge $E$}
\label{replace}
\end{figure}

If the singular edge $E$ has degree $=0$, take two colored  paths outgoing from $E$ and incoming to the same singular edge ( or  incoming to $E$ and outgoing from the same singular edge).  If  the degrees of the singular edges  on the colored paths  except the source and target vertices are not changed  by exchanging  the colorings of these colored paths, exchange the colorings of the colored paths. These conditions implies  that the source or the target singular edge which is not $E$ must not be positive type and the degree of its singular edge must be shifted $+1$, and  the degree of the singular edge $E$ is  changed to $- 1$. So we can replace the singular edge $E$ to the parallel normal edges and consider it as an element of $C(\Gamma^0)$.  We denote it as $\sigma' \in C(\Gamma^0) $. We only permit  color exchanging which satisfies  $\mathrm{gr} (\sigma)= \mathrm{gr} (\sigma ')$. 
 If there are some possibilities of exchanging the colorings, sum up all of them. If there are no possibilities of changing which satisfy the conditions, we consider the value as $0$.\\

\begin{figure}[htb]
\begin{center}
\includegraphics[scale=0.6]{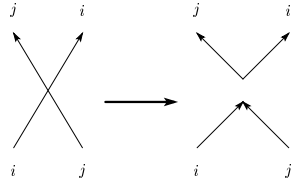}
\end{center}
\caption{Delete the singular edge whose grade=0}
\label{DelS}
\end{figure}

\begin{figure}[htb]
\begin{center}
\includegraphics[scale=0.6]{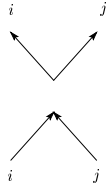}
\end{center}
\caption{Changing must be these type}
\label{AroundSingular}
\end{figure}

\begin{figure}[htb]
\begin{center}
\includegraphics[scale=0.6]{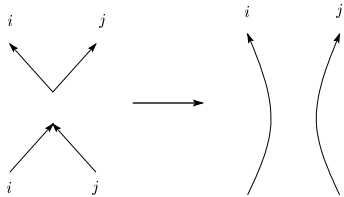}
\end{center}
\caption{Parallelize the edges which were connected to $E$}
\label{Redraw}
\end{figure}


In the $\chi_0$ case, we  would like to construct a  morphism $C(\Gamma^0) \longrightarrow   C(\Gamma^1)$ which is similar to $\chi _1$. Fix an element $\sigma \in C(\Gamma^0)$ and consider $\sigma$ as colored  oriented trivalent planar graph.  

Assume that the parallel normal edges of $\sigma$ have colorings  $i\geq j$    depicted as in the Fig.~\ref{chi0} and put a dotted circle around the parallel normal edges. This putted circle can be seen as a dotted line as in the left-hand side of the figure~\{fig}. We may treat a dotted circle like a singular edge.  
Take two colored paths which start from same singular edge (or the  dotted circle around the parallel edges) and end at same singular edge (or the  dotted circle around the  parallel normal edges). We assume that one of the colored paths contains the normal edge of parallel edges  colored $i,\ or\ j$.    These  two path can be seen  as a colored subgraph of $\sigma$ and we impose them  not to have self-intersections and  not to have a path which crosses  this subgraph whose intersection points with this subgraph are positive singular edges or negative singular edges. 
So  paths which traverse  the subgraph are allowed to have only degree zero intersection points. We also call  them as distinguished subgraph. (By considering the dotted circle as singular edge, a distinguished subgraph can be seen as a distinguished circle which is used in the $\chi_1$ case.)
Exchanging  colorings of  colored paths of a distinguished subgraph will introduce a new distinguished subgraph and  a new  colored oriented graph.

Take a distinguished circle of $\sigma$  whose source and terminal singular edges are not positive.
If exchanging  colorings  of the distinguished subgraph  does not change the  degrees of the singular edges on the distinguished subgraph   except  the source and target singular edges (or the dotted circle as source or target singular edge of distinguished subgraph), we will exchange the colorings of colored paths of  the distinguished subgraph and replace the parallel normal edges  to the singular  normal edge $E$ by the natural way as in the Fig.~\ref{replace}. We can   consider it as an element of $C(\Gamma^1)$. 
We  will modify and add other term as the special case. We must add and consider exceptional terms to $\chi_0$  which is related to the Reidemeister 1-move. 
If the normal edges of parallel edges colored $i$ or $j$ consist  a simple loop,   replace the parallel normal edges to singular edges  and change the coloring of simple closed curve as to preserve the grade as in the $\chi_1$ case.

If the subgraph does not permit these procedure (or there are no such subgraph satisfying the conditions), we send them to $0$. If there are some possibilities of taking distinguished subgraphs and exchanging the colorings, sum up all of them. 
\begin{figure}[htb]
\begin{center}
\includegraphics[scale=0.6]{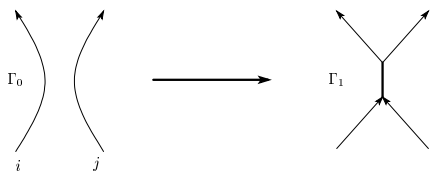}
\end{center}
\caption{$\chi_0$ morphism}
\label{chi0}
\end{figure}

\begin{figure}[htb]
\begin{center}
\includegraphics[scale=0.6]{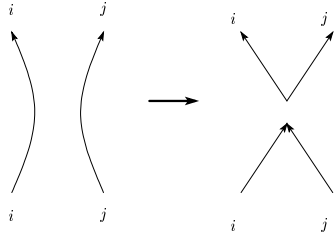}
\end{center}
\caption{Replace the parallel edges obtained by the 0-resolution}
\label{ReplacePara}
\end{figure}
Suppose the parallel edges of $\sigma$ colored $i,\ j$  depicted in the Fig.~\ref{chi0} has a relation $i<j$,  replace the parallel normal  edges to a singular edge $E$ and naturally extend its colorings. This singular edge $E$ has the degree $1$ and  we will consider them as an element of $C(\Gamma ^1)$.

So, $\chi_0$ and $\chi_1$ change some gradings of singular edges and colorings of normal edges. They are  degree$=0$ morphisms by  definitions.

\begin{proposition}
  The colored modules and the modules $( C(D(cr)),\ \chi_0 ,\  \chi_1)$  defined above  admit the cube structure.
\end{proposition}
\begin{proof}
To prove they have the cube structure, it is sufficient to confirm that  the structure morphisms commute.
Let $cr$  be a subset of $Cr$, and take two distinct  elements $a\ne b\in (Cr\setminus  cr)$. We denote the structure morphisms as  follows.
\[\xymatrix{
C(D(cr)) \ar[d]_{\chi_b} \ar[r]^{\chi_a} & C(D(cr\sqcup a))  \ar[d]^{\chi_{ab}} \\
C(D(cr\sqcup b)) \ar[r]^{\chi_{ba}}  & C(D(cr\sqcup a \sqcup b))\\
}
\]
Let  $a$  and $ b$ be negative crossings and  $E_a$ and $ E_b$ be the singular edges appearing in the $0-$resolutions of $a$ and $ b$. Then the structure morphisms $\chi_a,\ \chi_b,\ \chi_{ab}$ and $\chi_{ba}$ are all $\chi_1$ type morphisms.
If the distinguished circles of $E_a$ and $ E_b$ do not intersect  each other, $\chi_a$ and $ \chi_b$ do not affect to each other  because the $\chi_1$-type structure morphism changes only the local colorings of edges of distinguished colored circles or changes the colorings of   simple loops. If the singular edges $E_a$ and $ E_b$ have degree$=-1$, the morphisms $\chi_a$ and $ \chi_b$ are trivial.  Thus  in these cases, $\chi_{ab}\circ \chi_a = \chi_{ba}\circ \chi_{b}$. So, we assume that the distinguished circles have intersections.
 
By the definitions and the assumptions of  the distinguished circle and the structure morphism, the distinguished circles can have common singular edges.
But it is assumed that the structure morphism  does not change the degree of singular edges except  source and target singular edges and preserve the degree. If  the source or the target singular edge of the distinguished circle of $a$ is coincide to the singular edge $E_b$, they have common distinguished circle and the structure morphisms also commute.   Therefore exchanging procedures do not depend on the order.   Thus, they commute to each other. So we can conclude that the  compositions are commutative for $\chi_1$-type morphisms.
Suppose $a,\ b$ be positive crossings and $E_a,\ E_b$ be the singular edges appearing in the $1-$resolutions of $a,\ b$. In this case, we also exchange colorings of distinguished subgraph but this procedure is also commutative by the definitions and the assumptions of the distinguished subgraph and structure morphism. 
So, in this case, It is similarly confirmed that the structure morphisms commutes. 
The other cases can be checked by similar way.
So, we can conclude that the structure morphisms  commute to each other.
\end{proof}
Let $k$ be an integer. We denote by $\{ k \} $ the grading shift up by $k$. 
(i,e. Let $\sigma \in C(G)$ be a colored graph and $\mathrm{gr} (\sigma)=n$, then $\sigma\{ k \} \in C(G) \{ k \}$ has the grading  $\mathrm{gr} (\sigma)=n+k$.)
We denote the some propositions which are called MOY relations developed by Murakami, Ohtsuki and  Yamada \cite{5}.

\begin{proposition}
Let $\Gamma$ and $\Gamma_1$ be the graphs depicted in the figure~\ref{D1}. There is an isomorphisms as graded modules $C(\Gamma) \cong \oplus_{i=0} ^{n-2} C(\Gamma_1) \{  2-n+2i\}$.
\begin{figure}[htb]
\begin{center}
\includegraphics[scale=0.7]{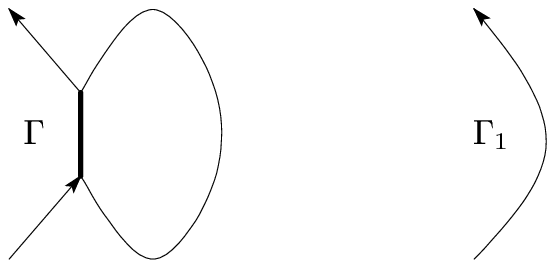}

\end{center}
\caption{}
\label{D1}
\end{figure}
\end{proposition}

\begin{proposition}
Let $\Gamma$ and $\Gamma '$ be the graphs depicted in the figure~\ref{DD2}. There is an isomorphism as graded modules $C(\Gamma ) \cong C(\Gamma ' )\{ 1\} \oplus C(\Gamma ' ) \{ -1\}$.
\begin{figure}[htb]
\begin{center}

\includegraphics{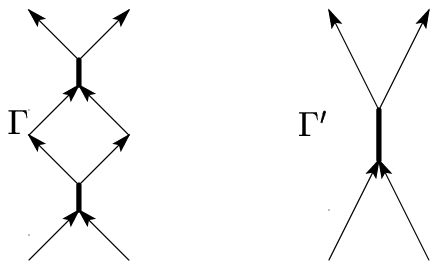}

\end{center}
\caption{}
\label{DD2}
\end{figure}
\end{proposition}
\begin{proposition}
Consider graphs $\Gamma_1,\ \Gamma_2,\ \Gamma_3,\ \Gamma_4$ depicted in the figure~\ref{BBBB}. There is an isomorphism as graded modules $C(\Gamma_1) \oplus C(\Gamma_2) \cong C(\Gamma_3) \oplus C(\Gamma_4)$.
\begin{figure}[htb]
\begin{center}

\includegraphics[scale=0.8]{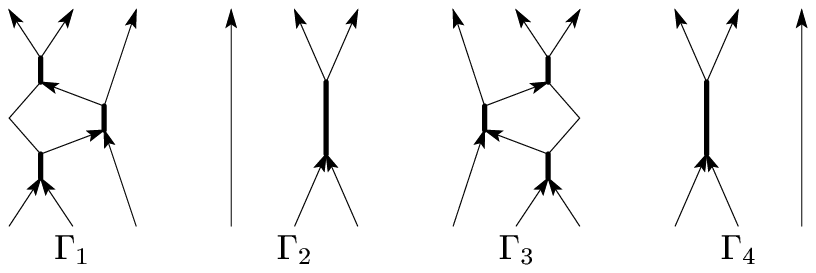}

\end{center}
\caption{}
\label{BBBB}
\end{figure}
\end{proposition}
These propositions are the  consequence  of  \cite{5}, derived from the representation of quantum group and its $R$-matrix , also have close connection to the theory of Kazhdan-Lusztig basis of Hecke algebras but we shall not pursue this line of thought.   We omit the proofs.

For a link diagram $D$ and $Cr\supset cr$ be a crossing set, we have constructed a complex associated  with $C(D(cr)), \ \chi_0,\ \chi_1$.
We assume to  locate $C(D(cr))$ at the cohomological degree$= N_1 ^+ - N_0 ^-$, and we will denote  
by $C^*(D)$ the complex associated with the diagram constructed by the shifted cube structure.
We state the main theorem whose proof is stated in the next section.

\begin{theorem}

\begin{itemize}
\item Let $D$ be an oriented link diagram  and $H\langle D \rangle$ be the
homology groups of $C^*(D)$.
Then $H\langle D \rangle$ is
invariant under the Reidemeister moves. 

\item The graded Euler characteristic of this homology is equal to the quantum polynomial invariant.
$$
\langle D \rangle _n= \sum_{i,\ j} { (-1)^{i} q^{j} rank(H^{i,\ j}(D))}
$$
where $i$ is the cohomological grading and $j$ is the grading of colored graphs.
\end{itemize}
\end{theorem}
\begin{proof}
The second statement follows from  the above propositions and \cite{5}.
\end{proof}

\section{Invariance under the Reidemeister moves}

We give the proof of the main  theorem stated above. These proofs  are similar to the one in \cite{2}. 
\subsection{R1-move}
Consider the type-I Reidemeister move.
\[ D=\xy
(0,10)*{}="T";
(0,-10)*{}="B";
(0,5)*{}="T'";
(0,-5)*{}="B'";
{\ar"T'";"T"};
"B";"B'" **\dir{-};
(3,0)*{}="MB";
(7,0)*{}="LB";
"T'";"LB" **\crv{(1,-4) & (7,-4)}; \POS?(.25)*{\hole}="2z";
"LB"; "2z" **\crv{(8,6) & (2,6)};
"2z"; "B'" **\crv{(0,-3)};
\endxy \ \ \ \ \ \ \ \ \
\ D' =
\xy
{\ar (0 , -10) ; (0 , 10)};
\endxy
\]

\begin{proposition}
For diagrams $D$ and $D'$ depicted above,  $C(D)$ is quasi-isomorphic to $C(D')$.
\end{proposition}

\begin{proof}
Let $a \in Cr(D)$ be a crossing of $D$ depicted above.  
We can consider  $\oplus _{ A\subset Cr(D),\ a\notin A } {C(D(A))}$ and $\oplus _{B\subset Cr(D),\ a\in B } {C(D(B))}$ as subcomplexes  of $C(D)$ whose differentials are induced from the differential of $C(D)$.
 Then $C(D)$  can be written as  the total complex of the bicomplex 
$$
\xymatrix{
0 \ar[r] & \oplus_{ A\subset Cr(D),\ a\notin A } {C(D(A))}  \ar[r]^{\chi} &  \oplus _{ B\subset Cr(D),\ a\in B } {C(D(B))} \ar[r] & 0
}
$$ where  
 the map  denoted $\chi$ is  induced by  the structure morphisms of cubes 
$ C(D(A)) \longrightarrow  C(D(a\sqcup A))$ , $A\subset Cr(D) $, $a\notin \ A $.

We will define  morphisms  $ f:\ C(D') \longrightarrow \oplus_{ B\subset Cr(D),\ a\in B } {C(D(B))} $ and    $ \epsilon :\ \oplus_{ B\subset Cr(D),\ a\in B } {C(D(B))}  \longrightarrow C(D')  $ as follows.

\begin{figure} [htb] 
\begin{center}
\includegraphics{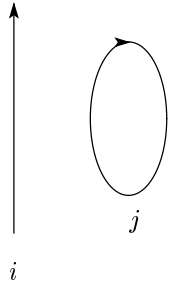}
\end{center}
\caption{An element of  1-resolution of the crossing $a\in Cr(D)$} 
\label{1-resR1}
\end{figure}

Define $\epsilon:$  $\oplus_{B\subset Cr(D),\ a\in  B} {C(D(B))}  \longrightarrow  C(D')$ to be a map which simply delete  the simple closed curve depicted in the Fig.~\ref{1-resR1} if its coloring is equal to $n$, otherwise  (if its coloring $j$ is not equal to $n$),  multiple $(-1)^{n-j}$ and take  $n-j$ distinguished circles which contain the normal edge colored $i$ depicted in the Fig.~\ref{1-resR1}, then exchange  colorings of each distinguished circles and  delete the simple closed curves depicted in the figure above and consider  as an element of $C(D')$.  Exchanging colorings of $(n-j)$ distinguished circles shift down its quantum degree $(n-j)$.  So the map $\epsilon $ is a quantum degree preserving morphism.\\ 
If we add a simple  circle which is colored $n$ to an element of $C(D')$ and consider it as an element of 1-resolution of $D$, we can consider this procedure  as a morphism  from $C(D')$ to $\oplus_{A\subset Cr(D),\ a\notin A} {C(D(A))}$ and we will denote $f$.

By the definitions of morphisms, we can check the formulas $\epsilon \circ  f = Id$,  $\epsilon \circ \chi =0$  and  we have a following formula. 
$$
\xymatrix{
C(D)\cong f(C(D')) \oplus \{ 0 \ar[r] & \oplus _{ A\subset Cr(D),\ a\notin A } {C(D(B))} \ar[r]  &  ker(\epsilon ) \ar[r]  &0 \}
}
$$     

Then we can say that the complex is isomorphic to the direct sum of the
contractible subcomplex
and the nontrivial part which  is isomorphic to 
the $C(D')$ by the formula $\epsilon f=Id$.  
This establishes  a homotopy equivalence between $C(D)$ and $C(D')$.
The invariance under other cases of Reidemeister-1 move can be verified similarly.
\end{proof}
\subsection{Reidemeister-2 move}
Consider diagrams $D$ and $D'$ depicted in 
the figure\ref{ta1}.
\begin{figure} [htb] 
\begin{center}
\includegraphics{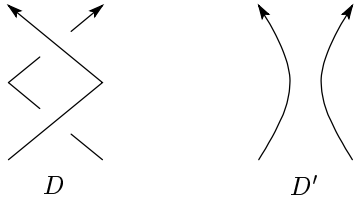}
\end{center}
\caption{Type IIa move} 
\label{ta1}
\end{figure}

\begin{proposition}
The complex $C(D)$ is quasi-isomorphic to
$C(D')$ for $D$ and $D'$ depicted in the Fig.~\ref{ta1}.
\end{proposition}
\begin{proof}

\begin{figure} [htb] 
\begin{center}
\includegraphics[scale=0.5]{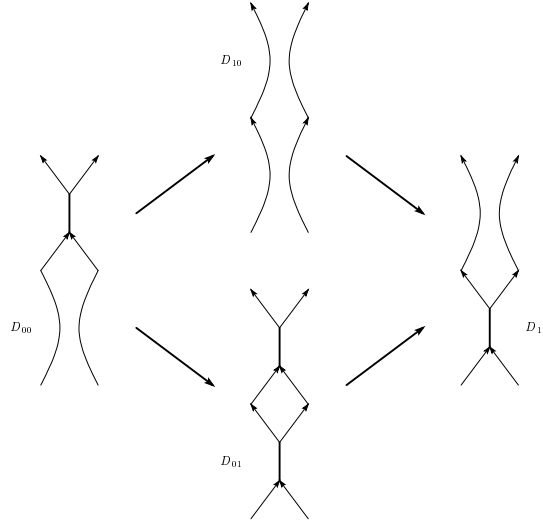}
\end{center}
\caption{Resolutions of $D$ } 
\label{reide2res}
\end{figure}

Fig.~\ref{reide2res} depicts the four resolutions of $D$ and morphisms between them.\\
Let $\psi $ be a morphism $\psi :\ C(D_{10}) \longrightarrow C(D_{00})$ which is defined by the  inverse procedure of the structure morphism $\chi_1:\ C(D_{00}): \ \longrightarrow C(D_{10})$ (i.e, take distinguished subgraphs and exchanging its colorings as to preserve the degree and sum up all of its possibilities and the  inverse of extra term. )  We define $\alpha $ to be a composition of $\psi$ and the structure morphism $\chi_0 :\ C(D_{}) \longrightarrow C(D_{01})$.\\
Let $Y_1=\{  (x,\alpha (x)) \in C(D_{10})\oplus  C(D_{01}),\ x\in C(D_{10})   \}$ .
We  can induce a differential on $Y_1$ by the differential on $C(D)$ and direct calculations shows that $Y_1$ is stable under the induced differentials. By the construction, it is obvious that $\chi_0 (x) + \chi_1 (\alpha (x)) =0 $ . \\
Let$Y_2$ be the subcomplex of $C(D)$ generated by $C(D _{00})$.\\
Let$Y_3= \{( \beta (x),\ y) \in C(D_{01})\oplus C(D_{11}),\ x,y\in C(D_{11})     \}$ where $\beta$ is a morphism $C(D_{11}) \longrightarrow C(D_{01})$ defined as the inverse procedure of the structure morphism $\chi_1:\ C(D_{11}) \longrightarrow C(D_{01})$.
As their construction, we can consider $Y_1$, $Y_2$ and  $Y_3$  as  the subcomlexes of $C(D)$.  
Then we can confirm directly that $C(D)$ is isomorphic to the direct sum $Y_1 \oplus Y_2 \oplus Y_3$ . The direct calculations of the differentials can also show that  $Y_2$ and  $Y_3$ are contractible complex by their construction and we can naturally calculate that  $Y_1$ is a subcomplex of $C(D)$  and there are natural morphism to $C(D')$ which is a quasi-isomorphism to $C(D')$. 
Then the  complex $C(D)$ can be described 
as $C(D)\cong  Y_1\oplus Y_2 \oplus Y_3$. This finishes the proof of this proposition.

\end{proof}
\subsection{Reidemeister-3}
\begin{figure} [htb] 
\begin{center}
\includegraphics[scale=0.6]{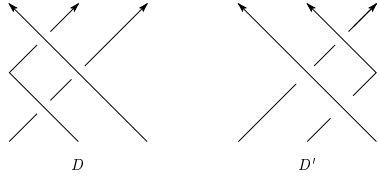}
\end{center}
\caption{Type III move} 
\label{braid1}
\end{figure}

Consider diagrams   $D,\ D'$ which is almost same but different only  in the small area depicted in the Fig.~\ref{braid1}. 

We will prove the following proposition. 
\begin{proposition}
The complex $C(D)$ is quasi-isomorphic to $C(D')$ for $D,\ D'$ depicted in the Fig.~\ref{braid1}.
\end{proposition}

\begin{proof}
The  complex  $C(D)$ is the total complex of the cube of  $8-$subcomplexes which is shown as in the Fig.~\ref{rthree1} and we  denote the colored graphs of the complex $C(D)$ as $C(\Gamma_{ijk}),$ for $i,\ j,\ k\in \{0,1\}$ in the 
figure~\ref{rthree1}.

\begin{figure} [htb] 
\begin{center}
\includegraphics[scale=0.7]{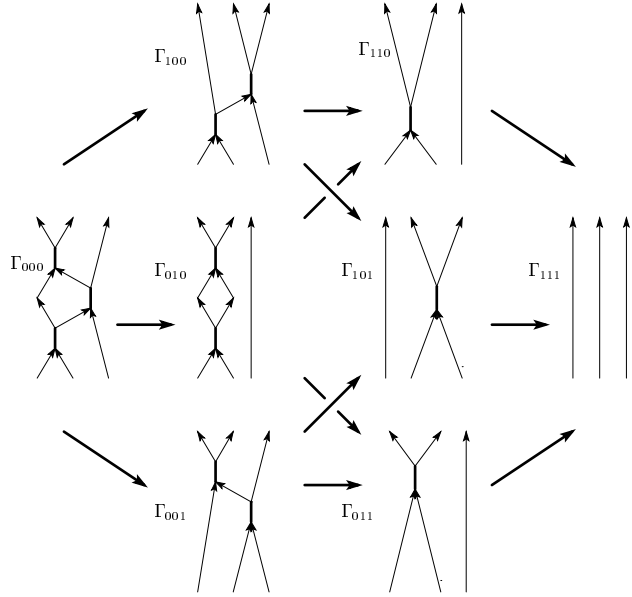}
\end{center}
\caption{Resolution cube of $D$} 
\label{rthree1}
\end{figure}

We   decompose $C(\Gamma_{000})$ into a direct sum of
two sub-modules.  Let $W$ be a submodule of $\Gamma _{000}$ genereted by
\begin{itemize}
\item $\Gamma_{\bullet \bullet \bullet }$  with $(k>j>i)$  and $(i>j>k)$.
\item $\Gamma_{\times \bullet \times }$ with$(j>i>k)$ and $(k>i>j)$.
\item $\Gamma_{\bullet \bullet \bullet }$+$\Gamma_{\times \bullet \times }$ with $(j<k<i)$ and $(i<k<j)$.
\item $\Gamma_{\times \bullet \bullet }$ with $(i>j>k)$ and $(k>j>i) $.
\item $\Gamma_{\bullet  \bullet \times }$ with $(j>i>k)$ and $(k>i>j)$.
\item $\Gamma_{\times \bullet \bullet }$+$\Gamma_{\bullet  \bullet \times }$  with $(j>k>i)$ and  $(i>k>j) $.
\item $\Gamma_{\bullet \times \bullet  } $ , $\Gamma _{\bullet \times \times  } $ ,   $\Gamma _{\times \times \bullet } $ and   $\Gamma _{\times \times \times } $ with all colorings.
\end{itemize}
where  $\Gamma_{\bullet \bullet \bullet }$ ,  $\Gamma_{\bullet  \bullet \times }$ ,  $\Gamma_{\bullet  \times \bullet }$ , $\Gamma_{\times \bullet \bullet }$ ,  $\Gamma _{\bullet \times \times  } $ ,   $\Gamma _{\times \bullet \times  } $ ,   $\Gamma _{\times \times \bullet } $ and   $\Gamma _{\times \times \times } $ are colored graphs  depicted in the  Fig.~\ref{direct-summand}.   Singular edges whose degree are not equal to $0$  will be depicted by black thick line and degree $0$ singular edges will be  depicted by simple crossing as in the Fig.~\ref{direct-summand}.

\begin{figure} [htb] 
\begin{center}
\includegraphics[scale=0.8]{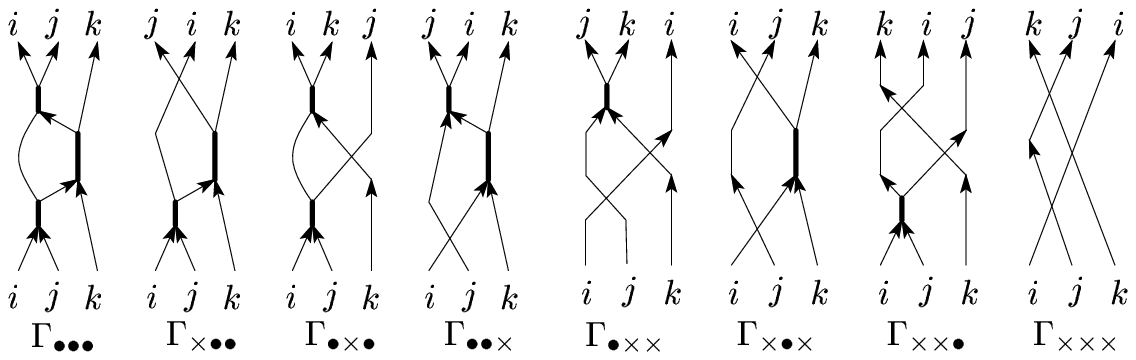}
\end{center}
\caption{colored graphs of $C(\Gamma _{000})$} 
\label{direct-summand}
\end{figure}

Let $\psi $ be a morphism $C(\Gamma _{001}) \longrightarrow C(\Gamma _{000})$ which is defined by the inverse procedure of the structure morphism $\chi_1$ and define $\alpha:\ C(\Gamma _{001}) \longrightarrow C(\Gamma _{010})$ to be the composition of $\psi  $ and structure morphism $\chi _1 :\ C(\Gamma_{000}) \longrightarrow C(\Gamma _{010}) $.\\

Using the submodule $W\subset C(\Gamma _{000}) $ and the morphism $\alpha$, we can decompose $C(D)$ as follows.

Let $Y_1$ be a submodule of $C(D)$  which consists of \\
 $\bullet$  $C(\Gamma _{1ij})$ , for $i,\ j\ \in \{ 0,\ 1\} $, \\
 $\bullet$  submodule  $W\subset C(\Gamma _{000})$ defined above,\\ 
 $\bullet$  $(x, \alpha (x) ),\ $ for $x\in C(\Gamma_{001})$ where the morphism $\alpha$ is defined above. \\
We can induce a differential on $Y_1$ by the differential on $C(D)$ and can check straightforward that they are stable under the induced differential. So $Y_1$ become a subcomplex of $C(D)$.\\

Let $Y_2$ be the  subcomplex of $C(D)$ which is generated by \\
   $\bullet $  $\Gamma_{\bullet \bullet \bullet }$ with $(k<j\ i<j)$ and $(j<i,\ j<k)$,\\
   $\bullet $  $\Gamma_{\times \bullet \times }$ with $(k<j<i)$ and $(i<j<k)$,\\
   $\bullet $  $\Gamma_{\bullet \bullet \times }$ with $(j<i,\ k<i)$ and $(i<j,\ i<k)$.  \\ 
   $\bullet $  $\Gamma_{\times \bullet \bullet}$ with $(j<i\leq k)$ and $(k\leq i<j)$.\\
   
Let $Y_3$  be a subcompex of $C(D)$ which is generated by\\
$\bullet $  $\Gamma_{\bullet \circ \bullet }$ with $(i>j)$,\\
$\bullet  $ $\Gamma _{\bullet \circ \times}$ with $(i<j)$, \\
$\bullet $ $\Gamma _{\times \circ \bullet }$ with $(i>j)$, \\  
$\bullet $ $\Gamma _{\times \circ \times }$ with $(i<j) $ \\
where $\Gamma _{\bullet \circ \bullet }$, $\Gamma _{\bullet \circ \times }$, $\Gamma _{\times \circ \bullet }$ and $\Gamma _{\times \circ \times }$ are colored graphs depicted in the Fig.~\ref{twistout}.\\

\begin{figure} [htb] 
\begin{center}
\includegraphics[scale=0.8]{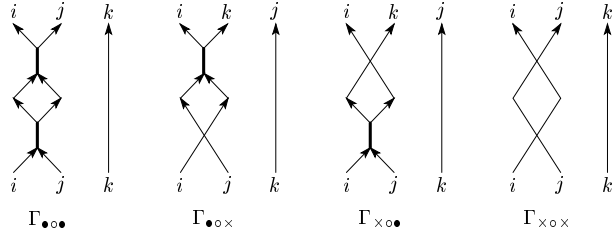}
\end{center}
\caption{generators of $\Gamma_{010}$}
\label{twistout}
\end{figure}

We can induce   differentials on $Y_1,\ Y_2,\ Y_3$ by a differential on $C(D)$   and consider them as subcomlexes of $C(D)$.
There is a natural isomorphism between $Y_1 \oplus Y_2\oplus Y_3$ and  $C(D) $ as abelian groups.
The direct calculation shows   that  $\Gamma_{\bullet \bullet \bullet }$  with $(i<j,\ k<j )\ (i>j,\ k>j)$ , $\Gamma_{\times \bullet \times }$ with $(i<j,\ i<k) \ (i>j,\ i>k)$ , $\Gamma_{\bullet  \bullet \times }$ with $(i<j,\ i<k)\ (i>j,\ i>k)$ and $\Gamma_{\times \bullet \bullet }$ with $(i<j,\ k<j)\ (i>j,\ k>j)$  are mapped  to $\Gamma_{101}$ injectively and this shows that  $Y_2$ is  acyclic complex. 
It is also confirmed directly that the $\Gamma_{\bullet \circ \bullet }$ with $(i>j)$,  $\Gamma _{\times \circ \times }$ with $(i<j)$,  $\Gamma _{\bullet \circ \times }$ with $(i<j)$ and $\Gamma _{\times \circ \bullet }$ with $(i>j)$  are mapped injectively to $\Gamma_{110} $ and this shows that $Y_3$ is acyclic complex. \\


Next, we will establish a similar decomposition.
The complex  $C(D')$  also has $8$ subcomplexes and denoted it as 
$\Gamma_{ijk} ' $ for $i,\ j,\ k\in \{0,1\}$ in the 
figure~\ref{rthree2}
\begin{figure} [htb] 
\begin{center}
\includegraphics[scale=0.7]{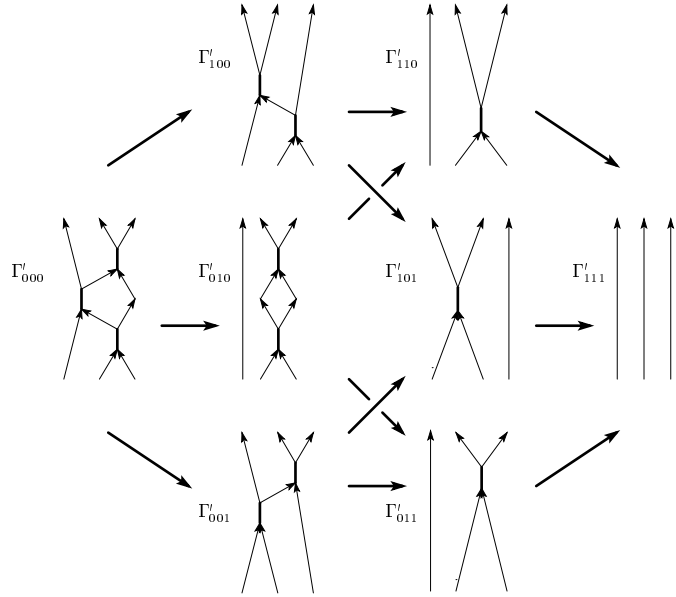}
\end{center}
\caption{Resolution cube of $D'$} 
\label{rthree2}
\end{figure}. 

A similar decomposition for $C(D')$ will also defined as follows.\\

\begin{figure} [htb] 
\begin{center}
\includegraphics[scale=0.8]{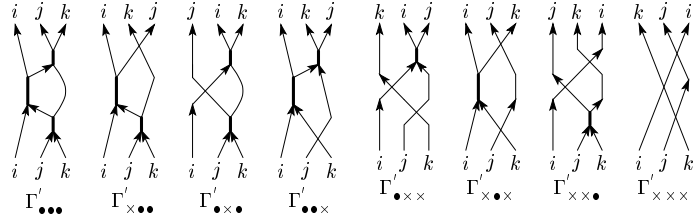}
\end{center}
\caption{ colored graphs $C(\Gamma _{000} ' )$} 
\label{direct-summand2}
\end{figure}

Let $W'$ be a submodule of $\Gamma _{000} '$ generated by 
\begin{itemize}
\item $\Gamma  _{\bullet \bullet \bullet } '$  with $(i> j>k )$  and $(k> j>i)$.
\item $\Gamma _{\times \bullet \times } '$ with$(j> k>i)$ and $(i>k>j)$.
\item $\Gamma  _{\bullet \bullet \bullet } ' + \Gamma _{\times \bullet \times } '$ with $(j>i>k)$ and $(k>i>j)$.
\item $\Gamma  _{\times \bullet \bullet } '$  with $(i> j>k )$  and $(k> j>i)$.
\item $\Gamma _{\bullet  \bullet \times } '$ with $(i>k>j)$ and $(j>k>i)$.
\item $\Gamma  _{\times \bullet \bullet } ' + \Gamma _{\bullet  \bullet \times } '$ with $(j>i>k)$ and $(k>i>j)$.
\item  $\Gamma _{\bullet \times \bullet } '$, $\Gamma _{\bullet \times \times  } '$,   $\Gamma _{\times \times \bullet } '$ and   $\Gamma _{\times \times \times } '$ with all colorings.
\end{itemize}
where  $\Gamma_{\bullet \bullet \bullet } '$, $\Gamma_{\times \bullet \bullet } ' $, $\Gamma _{\bullet \times \bullet } '$, $\Gamma_{\bullet  \bullet \times } '$ ,  $\Gamma _{\bullet \times \times  } '$,   $\Gamma _{\times \bullet \times  } '$,   $\Gamma _{\times \times \bullet } '$ and  $\Gamma _{\times \times \times } '$
are colored graphs  depicted in the  Fig.~\ref{direct-summand2}. \\
Let $\alpha '$ be a morphism $\Gamma _{100} ' \longrightarrow \Gamma _{010} '$ defined similarly as $\alpha$ by a composition of an inverse procedure of structure morphism $\chi_1 :\ \Gamma_{000} ' \longrightarrow \Gamma_{100} '$ and the  structure morphism  $\chi_1 :\ \Gamma_{000} ' \longrightarrow \Gamma_{010} '$ .\\

Let $Y_1 '$ be a subset of $C(D') $  which consists of \\
 $\bullet$  $\Gamma _{1 i j} '$ , for $i,\ j\ \in \{ 0,\ 1\} $, \\
 $\bullet$  submodule  $W ' \subset \Gamma _{000} '$ defined above,\\ 
 $\bullet$  $(x, \alpha ' (x) )$, for $x\in \Gamma_{100} '$ where the morphism $\alpha '$ is defined above. \\
We can induce a differential by the differential on $C(D')$ and  can check straightforward that $Y_1 '$ are stable under induced differential. So $Y_1$ become a subcomplex of $C(D')$.

Let $Y_2 '$ be a submodule of $C(D')$ which is generated by \\
 $\bullet $  $\Gamma_{\bullet \bullet \bullet } '$ with $(k<j,\ i<j)$ and $(j<i,\ j<k)$,\\
 $\bullet $  $\Gamma_{\times \bullet \times }$ with $(k<j<i)$ and $(i<j<k)$,\\
 $\bullet $  $\Gamma_{\bullet \bullet \times }$ with $( k<i,\ k<j)$ and $(i<k,\ j<k)$.  \\ 
 $\bullet $  $\Gamma_{\times \bullet \bullet}$ with $(i\leq k<j)$ and $(j\leq k<i)$.\\

\begin{figure} [htb] 
\begin{center}
\includegraphics[scale=0.8]{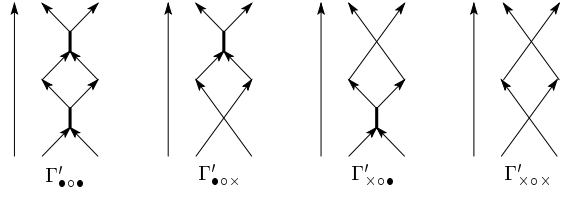}
\end{center}
\caption{generators of $\Gamma_{010} '$}
\label{twistout2}
\end{figure}

Let $Y_3 '$  be a subcompex of $C(D')$ which is generated by\\
$\bullet $  $\Gamma_{\bullet \circ \bullet } '$ with $(j>k)$,\\
$\bullet  $ $\Gamma _{\bullet \circ \times} '$ with $(j<k)$, \\
$\bullet $ $\Gamma _{\times \circ \bullet } '$ with $(j>k)$, \\  
$\bullet $ $\Gamma _{\times \circ \times } '$ with $(j<k) $ \\
where $\Gamma _{\bullet \circ \bullet } '$, $\Gamma _{\bullet \circ \times } '$, $\Gamma _{\times \circ \bullet } '$ and $\Gamma _{\times \circ \times } '$ are colored graphs depicted in the Fig.~\ref{twistout2}.

As in the $C(D)$ case, direct calculations of the differentials  show that  $Y_2 '$ and $Y_3 '$ are acyclic complexes.

Thus, we can twist out the acyclic complexes from $C(D)$ and $C(D')$. Therefore we can obtain a reduced complex $Y_1$ (resp. $Y_1 '$ )  from $C(D)$ ( resp. $C(D')$  ).

Next, we will make a correspondence graphically between $Y_1$ and $Y_1 '$ which induces a quasi-isomorphism between them.

For  $\Gamma_{\bullet \bullet \bullet }$  $(k>j>i)$, we associate to $\Gamma_{\bullet \bullet \bullet }' $   $(k>j>i)$ . We associate  for $\Gamma_{\bullet \bullet \bullet }$  $(i>j>k)$ to $\Gamma_{\bullet \bullet \bullet }'$  $(i>j>k)$.\\
For  $\Gamma_{\times \bullet \times }$ $(j>i>k)$ , we associate to $\Gamma_{\bullet \bullet \bullet } ' + \Gamma_{\times \bullet \times }'$   $(j>i>k)$.  We associate for     $\Gamma_{\times \bullet \times }$   $(k>i>j)$ to  $\Gamma_{\bullet \bullet \bullet }$+$\Gamma_{\times \bullet \times }$  $(k>i>j)$.\\
For  $\Gamma_{\bullet \bullet \bullet }+ \Gamma_{\times \bullet \times }$  $(j<k<i)$, we associate $\Gamma_{\times \bullet \times }'$  $(j<k<i)$. We associate for  $\Gamma_{\bullet \bullet \bullet }+ \Gamma_{\times \bullet \times }$   $(i<k<j)$ to  $\Gamma_{\times \bullet \times }'$  $(i<k<j)$ .\\
For $\Gamma_{\times \bullet \bullet }$  $(i>j>k)$  , we associate $\Gamma_{\bullet \times \bullet } '$  $(i>j>k) $. We associate for $\Gamma_{\times \bullet \bullet }$  $(k>j>i) $ to $\Gamma_{\bullet \times \bullet } '$   $(k>j>i)$. \\
For $\Gamma_{\bullet  \bullet \times }$  $(j>i>k)$, we associate $\Gamma_{\bullet \times \bullet } '$ $(j>i>k)$ . We associate for $\Gamma_{\bullet  \bullet \times }$ $(k>i>j)$ to $\Gamma_{\bullet \times \bullet } '$ $(k>i>j)$.\\
For $\Gamma_{\times \bullet \bullet }$+$\Gamma_{\bullet  \bullet \times }$   $(j>k>i)$, we associate  $\Gamma_{\bullet  \times  \bullet } '$  $(j>k>i)$. We associate for  $\Gamma_{\times \bullet \bullet }$+$\Gamma_{\bullet  \bullet \times }$ to $\Gamma_{\bullet t \times  \bullet } '$  $(i>k>j) $.\\
For  $\Gamma_{\bullet \times \bullet  } $ $(j>k>i)$ ,  we associate $\Gamma _{\bullet \bullet \times } '$  $(j>k>i)$.\\
For  $\Gamma_{\bullet \times \bullet  } $ $(i>k>j)$,   we associate $\Gamma _{\bullet \bullet \times } '$  $(i>k>j)$.\\
For  $\Gamma_{\bullet \times \bullet  } $ $(k>j>i)$,   we associate $\Gamma _{\times \bullet \bullet } '$  $(k>j>i)$.\\
For  $\Gamma_{\bullet \times \bullet  } $ $(i>j>k)$,   we associate $\Gamma _{\times \bullet \bullet } '$  $(i>j>k)$.\\
For  $\Gamma_{\bullet \times \bullet  } $ $(k>i>j)$,   we associate $\Gamma _{\bullet \bullet \times } ' + \Gamma _{\times \bullet \bullet } '$  $(k>i>j)$.\\
For  $\Gamma_{\bullet \times \bullet  } $ $(j>i>k)$,   we associate $\Gamma _{\bullet \bullet \times } ' + \Gamma _{\times \bullet \bullet } '$  $(j>i>k)$.\\
For $\Gamma _{\bullet \times \times  } $ , we can  associate naturally $\Gamma_{\times \times \bullet} '$.\\
For  $\Gamma _{\times \times \bullet } $ , we  can associate naturally $\Gamma _{\bullet \times \times} '$.\\
For  $\Gamma _{\times \times \times } $, we can  associate  naturally $\Gamma _{\times \times \times} '$ .\\
We can construct a natural correspondence between  $\Gamma _{1 i j}$ and $\Gamma _{1 i j}$, and also between $\{ (x,\ \alpha (x)) \  x \in  \Gamma _{001} \}$ and $\{ (y,\ \alpha ' (y)  ) \ y\in \Gamma _{100} ' \} $.

We can directly confirm that the above correspondence can  be seen as chain homomorphisms  $  Y_1 \longrightarrow \ Y_1 ' $ and   $  Y_1 ' \longrightarrow \ Y_1 $ and this chain homomorphisms give a homotopy equivalence . Therefore we conclude that our cohomology is invariant under the Reidemeister-move III
\end{proof}

This completes the proof that $H\langle G \rangle$
is  invariant under all the Reidemeister moves.

\section{Applications and generalizations}
We sketch here an application and  generalize  to a graphically defined infinite dimensional complex.
One application of  these graphical cohomology is to establish dualities between  links by constructing non-degenerate pairings.

\begin{theorem}
Let $L$ be a link. Denote by$-L$ the link obtained from $L$ by reversing the orientation of every component of $L$ and $\bar{L}$  the 
link obtained from $L$ by reversing the orientation of every component of $L$ and
switching the upper- and lower-branches at each crossing.  Then there are a  non-degenerate pairing
$H(L, k)\otimes _{k} H(\bar{L}, k ) \Longrightarrow k$ where $k$ means a field and a isomorphism $H(L,k)\cong H(-L,k)$.

\end{theorem}

\begin{proof}
Consider the chains $C(L),\ C(-L)$. Denote $D$ be a diagram of $L$ which is presented by a braid form.  By reversing the orientation of $L$, we can naturally make an isomorphism between $C(D)$ and $C(-D)$ as a graded module because reversing orientations does not change the crossing-type of $D$ and $-D$. But our structure morphism of cube does not depend on the orientations of underlying graphs. So the differentials also unchanged by reversing the orientations. Thus  we have an isomorphism $H(L,k)\cong H(-L,k)$.

$C(\bar{L} ,k)$ can be naturally identified to $Hom(C(L)),k$  where $k$ is a coefficient ring. By the standard homological algebra, we can identify the cohomologies of  $C(L,k)$ and $Hom(C(L),k)$. Thus we have a non-degenerate parings $H(L,k)\otimes H(\bar{L} ,k) \Longrightarrow k$.
\end{proof}

As a consequence we can conclude that the following formula. 
\begin{proposition}
Let $L $ be a link and $L'$ be a
link obtained from $L$ by  switching the upper- and lower-branches at each crossing.
If $L=\bar{L} $ or $L=L'$ , we have a non-degenerate paring $H(L, k)\otimes _{k} H({L}, k ) \Longrightarrow k$ 
\end{proposition}

We can consider a generalization of this graphically defined cohomology to infinite dimensional cohomology as the $HOMFLY$ polynomial can be obtained by generalizing the quantum $\mathfrak{sl}(n)$ polynomial invariant motivated by the skein theory.
The key observations is that the $HOMFLY$ polynomial is obtained by generalizing the quantum polynomial $\mathfrak{sl(n)}$ invariant based on the skein theory and that  the propositions  $3.6-3.7$, the [MOY]-relations, and the  proof of the invariance under Reidemeiser moves $II,\ III$ almost unaffected by the choice of $n$ which is the size of colorings (or the index of quantum group $\mathfrak{sl}(n)$ ).
We state here a non-modified definition of generalized complex but to construct a link homology, we must modify the definition. 

Let $B_n$ be a $n-$strands braid group and   $b_i \in B_n$ with $(1\leq i\leq n-1)$  be the standard generators and $b=(b_{i_1} ^{j_i} \cdots b_{i_k} ^{j_k})\in B_n$ be an element. We use graphical braid presentations as usual.
We will assign a complex to  $b$ which is invariant under the Reidemeister move II, III. 
We resolve the crossing of $b$ by the same rule of the $0,\ 1$ resolutions of crossings of a braid closure diagram. (See the Fig.~\ref{01-res}.)
Let $G$ be a oriented trivalent planar diagram obtained by the resolutions of each crossing of $b$. A state of $G$ is a function from the set of normal edges of $G$ to $\mathbb{N}$.
It should satisfy the following conditions that for each singular edge, the set of elements of $\mathbb{N}$ attached to the legs of the singular edge is equal to the set of elements of $\mathbb{N}$ attached to the heads of the singular edge and the elements attached to the legs (resp. heads) must be  different from each other. Furthermore the states must have same colorings at the initial  and the terminal points of the braid diagram(i.e., the state can be extended to  the closure of  braid  diagram). 
Given a state of $G$, we can take the  $n$-colored paths by considering the braid as the $n$-paths. If the strand has been attached $i$ by the state, we define the quantum degree of colored strand as $2i$. 
We define the  quantum degree of  a singular edge.
Let $E$ be a singular edge of $G$ appearing in the $1$-resolution of positive crossing as in figure~\ref{ws}. 
\begin{gather}
\mathrm{deg} (\ E \ )=
\begin{cases}
0 & \sigma (e_1)= \sigma (e_3)<\sigma (e_2) = \sigma (e_4), \\
1 & \sigma(e_1) = \sigma (e_4),\  \sigma (e_2) = \sigma (e_3), \\
2 & \sigma (e_1)= \sigma (e_3)>\sigma (e_2) = \sigma (e_4).
\end{cases}
\end{gather}

Let $E'$ be a singular edge of $G$ appearing in the $0$-resolution of negative crossing as in figure~\ref{ws}. 
\begin{gather}
\mathrm{deg} (\ E' \ )=
\begin{cases}
-2 & \sigma (e_1)= \sigma (e_3)<\sigma (e_2) = \sigma (e_4), \\
1 & \sigma(e_1) = \sigma (e_4),\  \sigma (e_2) = \sigma (e_3), \\
0 & \sigma (e_1)= \sigma (e_3)>\sigma (e_2) = \sigma (e_4).
\end{cases}
\end{gather}
This definition of quantum degree is the modified one defined in the definition $3.3$.
We also put  the quantum degree on the parallel normal edges appearing in the $0$-resolutions of positive crossings and  the $1$-resolutions of negative crossings. 
Let $P_0$, $N_1$ be th number of $0$-resolutions of positive crossings and $1$-resolutions of negative crossings respectively. Then we shift the quantum degree $2P_0 - 2N_1$.
This shift can be thought that  the $0$-resolutions of positive crossings have quantum degree $+2$ and the $1$-resolutions of negative crossings have quantum degree $-2$.
Then we define the quantum degree of $\sigma$ as 

$$
{deg (\sigma )}= (2P_0- 2N_1)+ \sum_{singular\ edges }deg(E)
++\sum_{i;\ colorings\ of\ strands} 2i
$$
Let $\tilde{ C(G)}$ be a graded  $\Z$-module whose generators are  the state of $G$ and the degrees are induced by the above  formula.
We also define the structure morphisms between them.  In this infinite dimensional graphical complex, we can also define the distinguished circles, the distinguished subgraphs and  the coloring exchange of distinguished circle or distinguished subgraph of $G$. The morphism corresponding to  $\chi_1$ is defined to  coloring exchanging  of the distinguished circle (without modification term) and the morphism corresponding to $\chi_0$ is defined to coloring exchanging of the distinguished subgraph.
It is confirmed similarly that the these graded modules and the  morphisms consist a cube. 
The propositions $3.6-3.7$  and the Reidemeister move$II,\ III$ also hold by the almost same arguments in the previous section. Thus we obtain a categorical braid representation.
  
\begin{theorem}
$\tilde{C(G)}$ and the structure morphisms define a chain complex which is invariant under the Reidemeister-moves II, III.
\end{theorem}
If we modify the $\tilde{C(G)}$, we can obtain a link homology which is invariant under the Reidemeister-moves I, II, III. But we do  not pursue these things  in this paper.

\section*{Acknowledgments}
Many people have made comments on earlier versions of the paper.
I would like to thank M. Harada, Y. Hashimoto, M. Masuda, A. Kawauchi and T.Tanisaki for useful discussions and comments.

\textsc{Osaka City University} 

\textit{E-mail address:} 
\textsf{d07sag0702@ex.media.osaka-cu.ac.jp}


\begin{thebibliography}{0}

\bibitem{1}
M. Khovanov, Categorification of the Jones polynomial,
{\it Duke J. of Math}. (2000), 359--426.

\bibitem{2}
M. Khovanov, L. Rosansky, Matrix factorizations and link homology,  {\it Fundamenta Mathematicae}. vol.199 (2008), 1-91.

\bibitem{3}
M. Khovanov, L. Rozansky,
Matrix factorizations and link homology 2,
{\it Geometry and Topology}. vol.12 (2008), 1387-1425.

\bibitem{4}
M. Khovanov, L. Rozansky,  Triply-graded link homology and Hochschild homology of Soergel bimodules, 
{\it International Journal of Math}. vol.18, no.8 (2007), 869-885.

\bibitem{5}
C.~Manolescu, Link homology theories from symplectic geometry,
{\it Advances in Mathematics}, Vol. 211 (2007), 363-416

\bibitem{6}
H. Murakami, T. Ohtsuki, S. Yamada, Homfly polynomial via an invariant of colored plane graphs, 
{\it Enseign.Math}.  (2) {\bf 44} (1998), 325--360. 

\bibitem{7}
V.G. Turaev,  The Yang-Baxter equation and invariants of links,
{\it Invent. Math}. {\bf 92} (1988), 527--553.

\bibitem{8}
H. Wu, Generic deformations of the colored $\mathfrak{sl}(n)$-homology for links.
preprint.



\end{thebibliography}
\end{document}